%% file: HF_arxiv.tex
\documentclass[12pt]{amsart}
\usepackage{amssymb,amsmath,graphicx,verbatim,eufrak,amsthm}
\usepackage{color}
\usepackage{xcolor}
\usepackage{mdframed}
\usepackage{bbm}

\usepackage{hyperref}
\hypersetup{
    linktoc=page,  
    colorlinks=true,       
    linkcolor=red,          
    citecolor=green,        
    filecolor=magenta,      
    urlcolor=cyan           
}

\definecolor{env_back}{gray}{0.8}
\definecolor{thm_color}{rgb}{0,0,1}
\definecolor{conj_color}{rgb}{1,0,0}
\definecolor{dfn_color}{cmyk}{0,1,0,0}

\newtheorem{thm}{{\color{thm_color}Theorem}}[section]
\newtheorem{cor}[thm]{{\color{thm_color}Corollary}}
\newtheorem{lem}[thm]{{\color{thm_color}Lemma}}

\newtheorem{prop}[thm]{{\color{thm_color}Proposition}}

\newtheorem{ques}[thm]{Question}
\newtheorem{conj}[thm]{{\color{conj_color}Conjecture}}

\newtheorem*{clm*}{Claim}

\theoremstyle{definition}

\newtheorem{exm1}[thm]{Example}

\theoremstyle{remark}

\newtheorem*{rem*}{Remark}

\newenvironment{lem*}[1]{\vspace{1ex}\noindent
{\bf Lemma* (#1).} [restatement]  \hspace{0.5em} \em }{ }
\newenvironment{thm*}[1]{\vspace{1ex}\noindent
{\bf Theorem* (#1).} [restatement]  \hspace{0.5em} \em }{ }

{\begin{exm1}}%
{\hfill$\bigtriangleup\bigtriangledown\bigtriangleup$
\end{exm1} }

\newcommand{\IP}[1]{\left\langle #1 \right\rangle}
\newcommand{\iP}[1]{\langle #1 \rangle}

\newcommand{\set}[1]{\left\{#1\right\}}

\newcommand{\sr}[1]{\left(#1\right)}

\newcommand{\Integer}{\mathbb{Z}}

\newcommand{\Z}{\Integer}
\newcommand{\N}{\mathbb{N}}

\newcommand{\R}{\mathbb{R}}
\newcommand{\C}{\mathbb{C}}

\newcommand{\eps}{\varepsilon}

\newcommand{\ie}{{\em i.e.\ }}
\newcommand{\eg}{{\em e.g.\ }}

\DeclareMathOperator{\E}{\mathbb{E}}     

\renewcommand{\Pr}{}
\let\Pr\relax
\DeclareMathOperator{\Pr}{\mathbb{P}}

\newcommand{\1}[1]{\mathbf{1}_{\set{ #1 } }}

\newcommand{\ind}[1]{\mathbf{1}_{ #1}}

\def\squareforqed{\hbox{\rlap{$\sqcap$}$\sqcup$}}
\def\qed{\ifmmode\squareforqed\else{\unskip\nobreak\hfil
\penalty50\hskip1em\null\nobreak\hfil\squareforqed
\parfillskip=0pt\finalhyphendemerits=0\endgraf}\fi}

\newcommand{\ignore}[1]{ }

\newcommand{\dist}{\mathrm{dist}}
\newcommand{\vphi}{\varphi}

\newcommand{\F}{\mathcal{F}}

\newcommand{\Ee}{\mathcal{E}}

\newcommand{\AND}{\qquad \textrm{and} \qquad}

\newcommand{\define}[1]{{\bf #1}}

\newcommand{\mn}{\wedge}

\newcommand{\spn}{\mathit{span}}

\newcommand{\ns}{\triangleleft}

\newcommand{\LHF}{\mathsf{LHF}}
\newcommand{\BHF}{\mathsf{BHF}}
\newcommand{\HF}{\mathsf{HF}}

\newcommand{\dom}{\backslash}

\newcommand{\LL}{\mathcal{L}}
\newcommand{\zz}{\mathbb{Z}}

\newcommand{\GG}{\mathbb{G}}

\newcommand{\FF}{\mathbb{F}}
\newcommand{\HH}{\mathbb{H}}
\newcommand{\RR}{\mathbb{R}}
\renewcommand{\AA}{\mathbb{A}}

\newcommand{\MSep}{\mathcal{MS}}

\newcommand{\clamt}[2]{
\bigl[
\begin{smallmatrix}
#2 & #1 \\  0 & 1
\end{smallmatrix}
\bigr]
}

\begin{document}

\title{Harmonic functions of linear growth on  solvable groups}

\author{Tom Meyerovitch \and Ariel Yadin}
\thanks{$^*$Ben Gurion University of the Negev. email: \texttt{\{mtom, yadina\} @math.bgu.ac.il}}
\thanks{ T.M.  would like to acknowledge funding from the People Programme (Marie Curie Actions) of the European Union's Seventh Framework Programme (FP7/2007-2013) under REA grant agreement no. 333598,
and from the Israel Science Foundation (grant no. 626/14). }

\address{Department of Mathematics, Ben Gurion University of the Negev, Be'er Sheva ISRAEL.}
\email{\{mtom, yadina\} @math.bgu.ac.il}

\maketitle

\begin{abstract}
In this work we study the structure of finitely generated groups for which a space of
harmonic functions with fixed polynomial growth is finite dimensional.  It  is conjectured that such groups must
be virtually nilpotent (the converse direction to Kleiner's theorem).  We prove that this is indeed the case
for solvable groups. 
The investigation  is partly motivated by Kleiner's proof for Gromov's theorem on groups of polynomial growth.
\end{abstract}

\setcounter{tocdepth}{1}
\tableofcontents

\input{HF_solv_introduction}

\input{HF_solv_conclusion}

\input{HF_solv_RW_subgroup}

\input{HF_solv_reduction_to_2x2}

\input{HF_solv_RW_real}

\input{HF_solv_2x2}

\end{document}

%% file: HF_solv_introduction.tex
\section{Introduction}
\label{scn: intro}

Based on Colding and Minicozzi's solution to Yau's Conjecture \cite{colding_minicozzi},
in 2007 Kleiner proved the following theorem \cite{kleiner_2010}: For any finitely generated group
$\GG$ of polynomial growth, the space of harmonic functions on $\GG$ of some
fixed polynomial growth
is a  finite dimensional vector space.  Using this theorem Kleiner obtained  a non-trivial finite dimensional representation of $\GG$,
 and  discovered a new proof of Gromov's theorem \cite{gromov_81}: Any finitely generated group of polynomial growth has a finite index subgroup
that is nilpotent (\ie is virtually nilpotent).

A natural question along these lines is whether the converse of Kleiner's theorem holds.
That is:

\begin{conj} \label{conj: main conj}
Let $\GG$ be a finitely generated group, and let $\mu$ be a symmetric probability measure
on $\GG$, with finite support that generates $\GG$.
Let $\HF_k(\GG,\mu)$ denote the space of $\mu$-harmonic functions on $\GG$
whose growth is bounded by a degree $k$ polynomial.

Then the following are equivalent:
\begin{enumerate}
\item $\GG$ is virtually nilpotent.
\item $\GG$ has polynomial growth.
\item $\dim \HF_k(\GG,\mu) < \infty$ for all $k$.
\item There exists $k \geq 1$ such that $\dim \HF_k(\GG,\mu) < \infty$.
\end{enumerate}
\end{conj}

Let us consider the space of {\em bounded} harmonic functions (\ie $\HF_0$), which we also denote by $\BHF= \BHF(\GG,\mu)$.
The space $\BHF(\GG,\mu)$ is isomorphic to the space of bounded functions on  the \emph{Poisson-Furstenberg boundary} of $(\GG,\mu)$.
This object has been studied extensively in the literature over the past.
We refer to the seminal paper \cite{kaimanovich_vershik_1983} of Ka{\u\i}manovich and Vershik
also \cite{MR2827814, Furman02} for background and more on this object.
From the theory of Poisson-Furstenberg boundaries it follows that when the Poisson-Furstenberg boundary is not one point, then it must be infinite.  So if the dimension of the bounded harmonic functions on $\GG$ is finite,
then the Poisson-Furstenberg boundary is finite, and thus trivial (a singleton).  Hence, the only bounded harmonic
functions in this case are the constants.  Positive harmonic functions
(Martin boundary, see \eg \cite{Sawyer97})  have also been extensively studied in the  literature.

Recently, there has been growing interest
in the study of spaces of unbounded harmonic functions on groups and other homogeneous spaces
(see \eg \cite{BDKY11} and references therein). As mentioned, Kleiner \cite{kleiner_2010} used the space $\HF_1$
of linearly growing functions to produce a finite dimensional representation for groups of polynomial growth,
which lead to a new proof of Gromov's theorem.

The main result of this paper, is a proof of Conjecture \ref{conj: main conj} in the case where
$\GG$ is a solvable group.
Our proofs and tools are probabilistic utilizing the theory of random walks on groups and martingales.

One consequence of our results is a structure theorem for the space of linearly growing harmonic functions
for general groups where this space is finite dimensional: 
Up to an additive constant and passing to a finite index subgroups any such function must be a homomophism into the additive group $\R$. In a follow-up paper joint with Idan Perl and  Matthew Tointon \cite{arXiv:1505.01175} we provide, along with additional  results, a structure theorem for the space of harmonic functions of polynomial growth (in the finite-dimensional case).

After introducing some notation, we will precisely state
the main contributions of this work in Section \ref{sec: main results}.

{\bf Acknowledgements:}
 The birth of this work was during a research seminar in Ben Gurion University.
 We thank the participants of the this seminar for their part.
  We acknowledge interesting conversations, encouragement and valuable suggestions from
 Amichai Aisenmann, Tsachik Gelander, Yair Glasner, Gady Kozma, Yehuda Shalom, Maud Szusterman  and the anonymous referees.

\subsection{Notation}

Throughout,
$\GG$ is a countable group generated by a finite symmetric set; $\GG = \IP{S},
S = S^{-1}, |S| < \infty$, and $\mu$ is a probability measure on $\GG$.
The generating set induces a metric $\dist_S$ on $\GG$, namely the graph metric of the Cayley graph
with respect to $S$. This metric is invariant to the action of $\GG$ from the left.
We write $|x| = |x|_S  = \dist_S(x,1)$ for $x \in \GG$.
The metrics obtained by different choices of generating sets are bi-Lipshcitz.

The pair $(\GG,\mu)$ is called a \define{measured group}. We will always assume
$\mu$ is a \define{symmetric} probability measure on $\GG$;
\ie $\mu(x) = \mu(x^{-1})$ for all $x \in \GG$, and that
$\mu$ is \define{adapted}: the support of $\mu$ generates $\GG$.
We say that $\mu$ is \define{smooth} if the generating function
$\vphi_\mu(\zeta) := \sum_{x} \mu(x) e^{\zeta |x|} $
has positive radius of convergence.

We say that $\mu$ is  \define{courteous} 
if it is a symmetric adapted and smooth probability measure.

Section \ref{sec:RW_finite_index_subgrp} details a bit the properties of smooth measures
and explains why the class of courteous measures is a natural class of probability measures to work with.

Clearly, any measure $\mu$ with finite support is smooth.
A primary example for a  courteous measure is the measure $\mu = \frac{1}{|S|}\sum_{s \in S}\delta_s$, distributed uniformly over a finite, symmetric generating set $S \subset \GG$.

A \define{$\mu$-harmonic function} $f:\GG \to \C$ is a function such that for all $x \in \GG$,
$f \ast \mu (x) : = \sum_s f(x s^{-1} ) \mu (s) = f(x) $.

A group $\GG$ acts naturally on functions on the group; namely by $xf(y) = f(x^{-1}y)$.

The space $\BHF = \BHF(\GG,\mu)$ of bounded harmonic functions  is a  $\GG$-invariant space; that is $\GG \BHF = \BHF$.

A measured group $(\GG,\mu)$ with the property that all bounded harmonic functions are constant is called \define{Liouville}.
The property of having a finite dimensional space of $\mu$-harmonic functions of linear growth can be viewed as a refinement of the Liouville property. 

We recall the following fact about bounded harmonic functions:
\begin{lem}
\label{lem: finite dimensional implies Liouville}
The only non-trivial finite-dimensional $\GG$-invariant subspace of $\BHF(\GG,\mu)$ is the constant functions;
that is if $V \leq \BHF$ is $\GG$-invariant and $\dim V < \infty$ then $V=\C$.

In particular, if $\dim \BHF < \infty$ then $(\GG,\mu)$ is Liouville.
\end{lem}

Lemma \ref{lem: finite dimensional implies Liouville} above is relatively classical in the study of Poisson-Furstenberg boundaries.
In private communication, Yehuda Shalom presented a slick argument based on the maximum principle and a compactness argument: The orbit closure of a non-trivial harmonic function with finite dimensional orbit contains a harmonic function with a proper maximum.

For a function $f:\GG \to \C$ define the polynomial-$k$-semi-norm:
\begin{equation}\label{eq:k_norm}
 || f ||_k : = \limsup_{r \to \infty} r^{-k} \cdot \max_{|x| \leq r} |f(x)| .
\end{equation}

We denote
\begin{equation}\label{ref:eq:HF_f_def}
\HF_k(\GG,\mu) := \set{ f:\GG \to \C \ \big| \ \|f\|_k < \infty \ , \ f \textrm{ is $\mu$-harmonic } }.
\end{equation}

The space $\HF_k = \HF_k(\GG,\mu)$ is  the space of $\mu$-harmonic functions
with polynomial growth of degree at most $k$.
Note that  $|| x f ||_k = || f ||_k$ for any $x \in \GG$, $f \in \C^{\GG}$, so the space $\HF_k$ is $\GG$-invariant.
The space $\HF_k(\GG,\mu)$
does not depend  on the choice of generating set for $\GG$ (but does inherently depend on the measure $\mu$).

Recall that a countable group $\GG$ is \define{amenable} if for any  $K \subset \GG$  and any $\epsilon >0$ there exists a finite  set  $F \subset \GG$ so that $|K F| \le (1+\epsilon)|F|$.
There are numerous equivalent definitions of amenability. For definitions and a detailed account of amenability for locally compact groups we refer for instance to \cite{pier_amenable}.

\begin{prop} \label{prop:liouville_is_amenable}
If $\dim \HF_k(\GG,\mu) < \infty$ for some $k \ge 0$, then $\GG$ is amenable.
\end{prop}

\begin{proof}
By Lemma \ref{lem: finite dimensional implies Liouville},  if $(\GG,\mu)$ is not  Liouville then  $\dim \BHF(\GG,\mu) =\infty$.

 Since $\BHF(\GG,\mu)=\HF_0(\GG,\mu)$ is a subspace of $\HF_k(\GG,\mu)$ for all $k \ge 0$, it follows that the assumption $\dim \HF_k(\GG,\mu) < \infty$ implies that $\GG$ is Liouville.

Rosenblatt \cite{rosenblatt_81} and independently Ka{\u\i}manovich-Vershik \cite{kaimanovich_vershik_1983}
showed that if $(\GG,\mu)$ is  Liouville then $\GG$ is amenable.
\end{proof}

By Proposition \ref{prop:liouville_is_amenable}, it suffices to consider amenable groups in
Conjecture \ref{conj: main conj}.

A \define{random walk} on $\GG$ with step distribution $\mu$ is a random sequence $(X_t)_t$
defined by $X_t = X_0 S_1 S_2 \cdots S_t$, where $(S_t)_t$ are i.i.d.-$\mu$. 
The canonical measure and expectation on $\GG^{\N}$ of this process, conditioned on $X_0=x$,
are denoted $\Pr_x , \E_x$.  When the subscript is omitted we mean $\Pr = \Pr_1 , \E = \E_1$
($1 = 1_{\GG}$ is the unit element in $\GG$).

Random walks and harmonic functions are intimately related.
$f$ is $\mu$-harmonic if and only if $(f(X_t))_t$ is a martingale.
See \cite[Chapter 5]{Durrett} for more on martingales.

\subsection{Statement of main results}\label{sec: main results}
Our main result is a proof of Conjecture \ref{conj: main conj} in the case that $\GG$ is a virtually solvable group. We recall that any virtually solvable group is in particular amenable, yet many amenable groups are not virtually solvable.

\begin{thm}
\label{thm: main thm}
Let $\GG$ be a
finitely generated virtually solvable group
and let $\mu$ be a courteous probability measure on $\GG$.
If $\dim \HF_k(\GG,\mu) < \infty$ for some $k \ge 1$, then $\GG$ is virtually nilpotent.
\end{thm}



Theorem \ref{thm: main thm} proves the implication $(4) \Rightarrow (1)$ of Conjecture \ref{conj: main conj} assuming $\GG$ is virtually solvable. 
All the other implications were previously known to hold without the assumption that $\GG$ is virtually solvable:
The implication $(1) \Rightarrow (2)$ is a standard computation and follows from the Bass-Guivarc'h formula \cite{bass_degree_growth_nilpotent,guivarch_degree_growth_nilpotent}. The implication
$(2) \Rightarrow (3)$ is by Kleiner \cite{kleiner_2010} via the method of Colding \& Minicozzi \cite{colding_minicozzi},
although strictly speaking, both Kleiner's proof
and the finitary version of Shalom and Tao \cite{shalom_tao_finitary_gromov} only deal with finitely supported measures.
The implication  $(3) \Rightarrow (4)$ is trivial.

A \textbf{linear group} is  one which can be embedded in a some group of matrices over a field.
Another direct but useful corollary of our result concerns linear groups:
\begin{cor}\label{cor:linear_finite_dim_HF_nilpotent}
Let $\GG$ be a finitely generated linear group
and let $\mu$ be a courteous probability measure on $\GG$. If  there exists $k \geq 1$ such that $\dim \HF_k(\GG,\mu) < \infty$  then $\GG$ is virtually nilpotent.
\end{cor}

\begin{proof}
Let $(\GG,\mu)$ be as above.
By Proposition \ref{prop:liouville_is_amenable}, $\GG$ is amenable. By the Tits alternative \cite{tits_1972}, an amenable finitely generated linear group is virtually solvable. It follows by Theorem \ref{thm: main thm} that $\GG$ is  virtually nilpotent.
\end{proof}

Our proof of Theorem \ref{thm: main thm} is based on the following theorem,
which is an explicit construction of a positive harmonic function of linear growth for finitely generated subgroups of the affine group.
Let $\FF$ be a field. The \define{affine group} of $\FF$ is defined by
\begin{equation}\label{eq:AA_def}
\AA(\FF) := \left\{ x \mapsto \lambda x + c~:~ \lambda \in \FF^{\times}~,~ c \in \FF \right\}
\end{equation}
(here $\FF^{\times}$ denotes the multiplicative group of invertible elements).

\begin{thm}\label{thm:2x2}
Let $\GG$ be a finitely generated subgroup of $\AA(\FF)$
which is not virtually abelian.
Suppose that $\mu$ is a courteous measure on $\GG$.
Then, there exists a positive, non-constant, $\mu$-harmonic function $f:\mathbb{G} \to [0,\infty)$
that has linear growth.

Moreover, the vector space spanned by the  orbit of $f$ under the $\GG$-action is infinite dimensional; \ie
$\dim \spn(\GG f) = \infty$.
\end{thm}

We prove Theorem \ref{thm:2x2} in Sections \ref{scn: RW on reals} and \ref{sec:2x2}.
Our construction of the function $f$  is a generalization of a particular construction of a positive harmonic function on the lamplighter group, which we now recall:
Let $\GG < \clamt{\mathbb{F}_p(x)}{\mathbb{F}_p(x)^\times}$ be the \define{lamplighter group with lamps in $\mathbb{F}_p$} which can be defined by:
$$\GG = \left\{
\clamt{c}{x^n}
\ : \
c \in \mathbb{F}_p[x,x^{-1}] \ , \  n \in \zz
 \right\}.$$

The  function $f_p$ is given by:
\begin{equation} 
f_p(x)= \lim_{k \to \infty} r_k \Pr_x \left[ c(\sigma_{r_k}) \in \mathbb{F}_p[x]\right],
\end{equation}
where:
\begin{itemize}
\item{$(r_k)_k$ is a certain increasing sequence of integers.}
\item{$\sigma_r = \inf \{t \geq 0 \ : \  \log |\lambda(t)| \not\in[-r,r]\}$, and $\lambda(t)$ denotes the upper-left entry of the random
$2 \times 2$ matrix $X_t$ (where $X_t$ is the random walk at time $t$).}
\item{$c(t)$ denotes the upper-right entry of the random $2\times 2$ matrix $X_t$.}
\end{itemize}

Subsequently to writing  a preliminarily version of our results, the construction above was exploited and  generalized in a different direction  by
Tointon \cite{tointon_harmonic} to characterize groups with the property that the space of {\em all} harmonic functions
is finite dimensional. Also, after finishing this paper we observed that a somewhat related construction of positive harmonic functions on affine groups appears in \cite{BBE, elie_bougerol}. However, it is not immediately clear if the results of  \cite{BBE, elie_bougerol} can be directly applied. One missing piece is  an estimate for the polynomial growth rate of those functions. 

In Section \ref{sec:reduction}  we carry out a reduction of Theorem \ref{thm: main thm} to the case where
$\GG$ is a subgroup of $\AA(\FF)$  as in the assumption of Theorem \ref{thm:2x2}.
For this reduction we invoke a theorem of Groves \cite{Groves78} regarding finitely generated solvable groups,
see Theorem \ref{thm: Groves}.

%% file: HF_solv_conclusion.tex
\section{Further research directions and open questions}
\label{sec:conclusion}

Before going into the proofs, let us mention some further research directions and some conclusions
from this work.

\subsection{Consequences of  Conjecture \ref{conj: main conj}}
Much is known about random walks on  finitely generated group $\GG$ of polynomial growth,
see \eg \cite{Alex02, HS-C93}.  For example, the random walk is {\em diffusive};
that is, $\E [ |X_t|^2 ] \asymp C t$.  Another example regards the {\em entropy} of the random walk:
since a ball has polynomial growth, we have that $H(X_t) = O(\log t)$
(for more on entropy and random walks see \cite{kaimanovich_vershik_1983} and references therein).

Thus, the following are  also consequences of Conjecture \ref{conj: main conj}:
\begin{conj}
Let $\mu$ be a courteous probability measure on $G$ such that $\dim \HF_1(\GG,\mu) < \infty$.
Then:
\begin{itemize}
\item $\E [ |X_t|^2 ] = O(t)$ where $(X_t)_t$ is a $\mu$-random walk on $\GG$.
\item $H(X_t) = O(\log t)$ where $H(\cdot)$ is the Shannon entropy of a discrete random variable.
\end{itemize}
\end{conj}

Conjecture \ref{conj: main conj} implies that if one space $\HF_k$ is finite dimensional, then all of them are:

\begin{conj}
Suppose that $\mu$ is a courteous probability measure on a finitely generated group $\GG$ such that
there exists $k$ with $\dim \HF_k(\GG,\mu) < \infty$.

Then, for {\em every} $k$ we have that $\dim \HF_k(\GG,\mu) < \infty$.
\end{conj}


A well-known open question in the subject is whether the Liouville property is a group invariant. We must restrict ourselves to a certain
class of measures:  Indeed, for the lamplighter group on $\Z$, $\LL(\Z) = (\Z/2\Z) \wr \Z$,
it is quite simple to construct finitely-supported but non-symmetric measures and infinitely supported symmetric
but non-smooth measures which are  non-Liouville.
However, any courteous measure on  $\LL(\Z)$ is Liouville \cite[Section 6.3]{kaimanovich_vershik_1983}.
(Actually \cite{kaimanovich_vershik_1983} proves this for finitely supported symmetric adapted measures.
The proof for courteous measures is along the same lines, and does not require any new ideas. One may also
use a coupling argument in the spirit of \cite{BDKYprep}.)

The following conjecture has been unresolved for quite some time:

\begin{conj} \label{conj: Liouville invariance}
Let $\mu,\nu$ be two courteous probability measures on a finitely generated group $\GG$.
Then $(\GG,\mu)$ is Liouville if and only if $(\GG,\nu)$ is Liouville.
\end{conj}


Regarding harmonic functions of polynomial growth, the analogous question is: 

\begin{conj}\label{conj:dim_HF_k_invariance}
Let $\mu,\nu$ be two courteous probability measures on a finitely generated group $\GG$.
Then, for any $k$,
$$ \dim \HF_k(\GG,\mu) = \dim \HF_k(\GG,\nu) . $$
\end{conj}

This conjecture has been verified for the class of virtually solvable groups in \cite{arXiv:1505.01175}, building on  results in the current paper, or more generally under the assumption that $\HF_k(\GG,\mu)$ and $\dim \HF_k(\GG,\nu)$ are both finite.
Note that Conjecture \ref{conj: Liouville invariance} is just the $k=0$ case of Conjecture \ref{conj:dim_HF_k_invariance},
because $\dim \BHF(\GG,\mu)$ is either $1$ or $\infty$.

Progress toward proving Conjecture \ref{conj: main conj} will probably require an understanding of the 
kernel of the $\GG$-action on $\HF_k(\GG,\mu)$.
For example, if this kernel is trivial for some large enough $k$, the group is linear, and our results
hold in this case.

\begin{ques}
Suppose that $\mu$ is a courteous measure on a finitely generated group $\GG$.
Describe the kernel of the $\GG$-action on $\HF_k(\GG,\mu)$.
\end{ques}

\subsection{Locally compact metric groups}

Throughout this paper we considered finitely generated countable groups. However, the definition of the spaces $\HF_k(\GG,\mu)$ can be formulated for any measured group admitting a left-invariant metric,
and it is natural to attempt to extend our results to the more general setting.
The reader may verify that the proof of Theorem \ref{thm:2x2} does not actually require $\GG$ to be finitely generated,
rather that it should be equipped with a left invariant metric for which the conclusion of Lemma \ref{lem:alpha_distance} holds.
It seems plausible that the reduction from Theorem \ref{thm:2x2} to Theorem \ref{thm: main thm} should hold for groups admitting a left-invariant metric in some greater generality.
For instance, a topological version of the Tits alternative
is known \cite{glander_breuillard_topological_tits}.

Some care is required when the word metric for a finitely generated group is replaced by an arbitrary left-invariant metric. To illustrate the point, consider the group $\GG=\zz$,
with respect to the invariant metric $d(n,m) = \sqrt{| n-m|}$.  It is no long true that homomorphisms are Lipschitz with respect to this metric.

\subsection{Lipschitz harmonic functions}

We believe the positive harmonic function $f$ appearing in Theorem \ref{thm:2x2} is in fact Lipschitz.
The probability estimates required to prove this seem to be more delicate then those appearing in Section \ref{sec:2x2}.
A proof that $f$ is Lipschitz would allows us to conclude that any finitely generated solvable group $\GG$ such that  the space of {Lipschitz harmonic functions is finite dimensional is virtually nilpotent, slightly improving Theorem \ref{thm: main thm}.
Note that Kleiner's strategy to prove Gromov's Theorem only requires that $\dim \LHF(\GG,\mu) < \infty$.

\begin{conj}
Let $\mu$ be a courteous probability measure on $\GG$. If
$\dim \LHF(\GG,\mu) < \infty$ then $\dim \HF_1(\GG,\mu) < \infty$.
\end{conj}

%% file: HF_solv_RW_subgroup.tex
\section{Reduction from solvable  to subgroups of $\AA(\FF)$}
\label{sec:reduction}

\subsection{Random walks and finite index subgroups}\label{sec:RW_finite_index_subgrp}
The nature of our result forces us to pass to finite index subgroups in the course of the proof.
In this section we review the basic correspondence between harmonic functions on a group and on a finite index subgroup.

In this section $\GG$ is a general finitely generated group,  $(X_t)_t$ is a random walk on $\GG$ with jump
distribution $\mu$, where $\mu$ is a courteous probability measure on $\GG$.
For a subgroup $\HH < \GG$,
define the \define{hitting time} 
$\tau_\HH = \inf \set{ t \geq 1 \ : \ X_t \in \HH }$.
We say that $\HH$ is a \define{recurrent subgroup} of $\GG$ if $\tau_\HH < \infty$ a.s.
It is well known that a subgroup of finite index is always recurrent. Furthermore, the expectation of $\tau_\HH$ is equal to $[\GG:\HH]$ (see \eg \cite{HLT14} for a development of such relations in this context).

A random variable $X$ has an \emph{exponential tail} if $\Pr[ |X| >t ] < C e^{-ct}$ for some constants $C,c>0$. It is straightforward to see that
$X$ has an exponential tail if and only if
$\E[e^{\alpha |X|}] < \infty$ for some $\alpha >0$.

We need the following observation:
\begin{lem}\label{lem:return_time_expoenetial}
Let $[\GG:\HH] < \infty$.
For any adapted measure $\mu$, $\tau_\HH$ has an exponential tail.
\end{lem}

\begin{proof}
Because $\mu$ is adapted,
$(\HH X_t)_t$ is an {\em irreducible} Markov chain on the finite set $\HH \dom \GG$.
$\tau_\HH$ is the first time this chain returns to the coset $\HH$.
For any irreducible Markov chain on a finite set, hitting times always have an exponential tail.
\end{proof}

For  a recurrent subgroup $\HH$ the \define{hitting measure} is  the
a probability measure on $\HH$ defined by
$$\mu_\HH(x) = \Pr_1 [ X_{\tau_{\HH}} = x ].$$

Note that a measure $\mu$ on a metric group $\GG$ is \emph{smooth}  
if and only if the length of a $\mu$-random element of $\GG$ has an exponential tail.

\begin{lem}\label{lem:finite_index_preserves_exponential_decay}
Let $\mu$ be an adapted smooth measure on $\GG$ and
$\HH \leq \GG$ a subgroup of finite index.
The hitting measure $\mu_{\HH}$ is also a smooth measure.
\end{lem}


\begin{proof}
Let $Z=Z_\HH := S_1 \cdot S_2 \cdots S_{\tau_\HH}$, 
with $S_1,S_2,\ldots$
independently and identically distributed according to $\mu$.
Note that for non-negative random variables, the property of having exponential tail is
monotone with respect to stochastic domination.
Clearly, $|Z| \le \sum_{k=1}^{\tau_\HH} |S_k|$. However, $\tau_\HH$ is not necessarily independent from $(|S_t|)_{t=1}^{\infty}$.
We overcome this dependence as follows:

Let $x_1,\ldots,x_N \in \GG$
be a set of representatives for right $\HH$-cosets with $x_1 = 1_\GG$,
so $\GG = \biguplus_{j=1}^N \HH x_j$ where $N = [\GG:\HH]$.
Define a family of jointly independent  $\GG$-valued random variables $$(S_t(i,j) \ : \ t \in \N \ , \ i,j \in \set{1,\ldots,N} )$$
as follows: $S_t(i,j)$ is distributed according to $\mu$ conditioned on the event that
$S_t \in x_i^{-1} \HH x_j$
(if this event has zero measure with respect to $\mu$ we make the arbitrary definition $S_t(i,j)=1$).
Verify that $Z$ is equal in distribution to a random variable of the following form:
$$\prod_{k=1}^{\tau_\HH} S_{k}(\xi_{k-1},\xi_{k}),$$
Where $\xi_0=1$ and $\xi_1,\xi_2,\xi_3$ are a sequence of random variables taking values in $\set{1,\ldots,N}$, whose law is determined by the finite state Markov chain
$$\Pr[\xi_{t+1} = j \ | \   \xi_t=i] = \mu\left( x_i^{-1} \HH x_j \right).$$

Since $(S_t)_t$ all have an exponential tail, it follows that $(S_t(i,j))_{t,i,j}$ also have an exponential tail.
Let $\hat{S}_t := \max_{i,j} |S_t(i,j)|$. Then $\hat{S}_t$ also has an exponential tail, and
$|Z|$ is stochastically dominated by $\sum_{t=1}^{\tau_\HH} \hat{S}_t$.
Now, note that $\tau_\HH$ is equal in distribution to $\inf \set{ t\geq 1 \ : \ \xi_t = 1 }$, and also,
$(\hat{S}_t)_t$ are independent of $(\xi_t)_t$.
Since $\tau_\HH, \hat{S}_1, \hat{S}_2,\ldots$
have exponential tails and are jointly independent, it follows that
$$ |Z| \textrm{ is stochastically dominated by } W : = \sum_{t=1}^{\tau_\HH} \hat S_t , $$
where $(\hat S_t)_t ,\tau_{\HH}$ are all independent $\Z_+$-valued random variables,
$(\hat S_t)_t$ are i.i.d  and all have an exponential tail.

We are left with showing that there exists $\alpha > 0$ such that $\E [ e^{\alpha W} ] < \infty$.

We know that since $(\hat S_t)_t$ all have an exponential tail, there exists $\beta > 0$ such that $\E[e^{\beta \hat S_t} ] < \infty$.
Similarly for $\tau_\HH$, there exists $\gamma >0$ such that $\E [ e^{\gamma \tau_\HH} ] < \infty$.
Dominated convergence guaranties that $\E [e^{\beta \hat S_t} ]$ is
continuous in $\beta$ and $\E [e^{\beta \hat S_t} ] \to 1$ as $\beta \to 0$.
Thus, we may choose $\alpha > 0$ small enough so that $\E [ e^{\alpha \hat S_t} ] < e^{\gamma}$.
With this choice we have by independence of $\tau_\HH, (\hat S_t)_t$,
\begin{align*}
\E [ e^{\alpha W} ] & = \sum_{k=0}^{\infty} \Pr[ \tau_\HH = k ] \cdot \prod_{t=1}^k \E [ e^{\alpha \hat S_t} ]
\leq \sum_{k=0}^{\infty} \Pr[ \tau_\HH = k ] \cdot e^{\gamma k}  = \E[ e^{\gamma \tau_\HH} ] < \infty .
\end{align*}
\end{proof}

Another property we wish to explore is the relation between $\HF_k(\GG,\mu)$ and $\HF_k(\HH,\mu_\HH)$.

\begin{lem} \label{lem: moments of Xt}
Let $\mu$ be a courteous probability measure on $\GG$.
Let $(X_t)_t$ be a $\mu$-random walk on $\GG$. 
Then, for any $k$ there exists a constant $C = C_k$
such that for every $t \geq 0$,
\begin{equation}\label{eq:X_t_moment}
 \E [ |X_t|^k] \leq C_k \cdot  t^k . 
\end{equation}

Consequently, for any $x \in \GG$,
\begin{equation}\label{eq:X_t_moment_x}
\E_x [ |X_t|^k ] \leq C_k \cdot (t+|x|)^k . 
\end{equation}
\end{lem}

\begin{proof}
By the triangle inequality we have 
$|X_t| \le |X_0|+ \sum_{j=1}^t |S_j|$, where $S_j = X_{j-1}^{-1} X_j$ is the jump at time $j$.
So
\begin{equation}\label{eq:X_t_moment2}
\E [ |X_t|^k]   \le \E \left[ \left(\sum_{j=1}^t |S_j|\right)^k\right] =
 \sum_{\vec{j} \in \{1,\ldots,t\}^k} \E [\prod_{i=1}^k |S_{j_i}|].
\end{equation}
Let $$ C_k = \max_{ 1 \le n \le k} \left(\E [|S_1|^n]\right)^{k/n} . $$
It follows that
\begin{equation}\label{eq:X_t_moment3}
\E [\prod_{i=1}^k |S_{j_i}|] \le C_k \mbox{ for all } \vec{j} \in \{1,\ldots,t\}^k.
\end{equation}

Thus \eqref{eq:X_t_moment} follows from \eqref{eq:X_t_moment2} and \eqref{eq:X_t_moment3}.

To prove \eqref{eq:X_t_moment_x},  note that
\begin{align*}
\E_x |X_t|^k & = \E |xX_t|^k \leq \E [ (|x|+|X_t|)^k ] \\
& = \sum_{j=0}^k {k \choose j} |x|^j \cdot \E |X_t|^{k-j} \leq \sum_{j=0}^k {k \choose j} |x|^j \cdot C_{k-j} t^{k-j} .
\end{align*}
Taking $C'_k = \max_{j \leq k} C_j$ we get that
$$ \E_x |X_t|^k \leq C'_k \cdot (t+|x|)^k . $$
\end{proof}

The following lemma shows that $\mu$-harmonic functions on $\GG$ correspond bijectively to $\mu_\HH$-harmonic functions on $\HH$:
\begin{prop}\label{prop:restriction_of_harmonic}
Let $\GG$ be a finitely generated group, $\mu$ a courteous measure on $\GG$ and $\HH \leq \GG$ a subgroup of finite index. For any $k  \ge 0$, the restriction of any $f \in \HF_k(\GG,\mu)$ to $\HH$ is $\mu_{\HH}$-harmonic and in $\HF_k(\HH,\mu_{\HH})$. Conversely, any $\tilde f \in \HF_k(\HH,\mu_\HH)$ is the restriction of a unique $f \in \HF_k(\GG,\mu)$.
Thus, the restriction map is a linear bijection from  $\HF_k(\GG,\mu)$ to $\HF_k(\HH,\mu_\HH)$.
\end{prop}

\begin{proof}
{\bf Step I: Extension.}
Let $\tilde f \in \HF_k(\HH,\mu_{\HH})$.  Define $f:\GG \to \C$ by
$$ f(x) : = \E_x [ \tilde f (X_{\tau }) ] $$
where $\tau = \tau_\HH$ is the return time to $\HH$ and $(X_t)_t$ is a $\mu$-random walk on $\GG$.

We now wish to show that $f$ is well-defined (equivalently, the expectation $\E_x [ |\tilde f (X_{\tau })|]$ is finite), and that  $f \in \HF_k(\GG,\mu)$.

Note that since $[\GG : \HH] < \infty$ and $\GG$ is finitely generated so is $\HH$.
Also,  with resect to any choices of finite symmetric generating sets $\GG = \IP{S}$ and $\HH = \iP{\tilde S}$ the corresponding metrics
are bi-Lipschitz. Namely, there exist $C>1$ so that for any $x \in \HH$ we have that $C^{-1} |x|_{S} \leq |x|_{\tilde S} \leq C \cdot |x|_{S}$ (see for instance  Corollary $24$ on page $89$ of \cite{de_la_harpe_geom_group_book}).
Thus, since $X_\tau \in \HH$ and $\tilde f \in \HF_k(\HH,\mu_{\HH})$,
we have that $|\tilde f(X_\tau) | \leq C \cdot |X_\tau|^k$ for some constant $C>0$.

Now, for any $x \in \GG$,
since $\set{ \tau > t } = \set{ \tau \leq t }^c \in \F_t : = \sigma(X_0,\ldots,X_t)$,
and since $S_{t+1} : = X_t^{-1} X_{t+1}$ is independent of $\F_t$,
we get that
\begin{align*}
|f(x)| & \leq C \cdot \E_x [ |X_{\tau} |^k ] = C |x|^k + C \cdot \sum_{t=0}^{\infty} \E_x [ \1{ \tau > t } \cdot (|X_{t+1}|^k - |X_t|^k)]  \\
& \leq C |x|^k + C \cdot \sum_{t=0}^\infty \E_x [ \1{ \tau>t} \cdot ( (|X_t| + |S_{t+1}|)^k - |X_t|^k ) ]
\\
& = C |x|^k + C \cdot \sum_{t=0}^{\infty} \sum_{j=1}^k {k \choose j} \E [|S_{t+1}|^j] \cdot \E_x [ \1{\tau > t } |X_t|^{k-j} ] .
\end{align*}
Using Lemma \ref{lem: moments of Xt} we have that there exists some constant $C_k>0$ such that
\begin{align*}
\big( \E_x [ \1{\tau > t } |X_t|^{k-j} ] \big)^2 & \leq \Pr_x [ \tau > t ] \cdot \E_x [ |X_t|^{2(k-j)} ] \\
& \leq \Pr_x [\tau >t ] \cdot C_k \cdot (|x|+t)^{2(k-j)} .
\end{align*}
Since $\tau$ has an exponential tail, setting $M_k = \max_{j \leq k} \E [|S_{t+1}|^j]$, we have
\begin{align*}
|f(x)| & \leq C |x|^k + C \cdot C_k \cdot M_k \cdot \sum_{t=0}^{\infty} e^{-c t} \sum_{j=1}^k {k \choose j} (|x|+t)^{k-j} \\
& = C |x|^k + C \cdot C_k \cdot M_k \cdot \sum_{t=0}^\infty e^{-ct} ((|x|+t+1)^k - (|x|+t)^k)  = O(|x|^k) .
\end{align*}
This proves that $f$ is well-defined and that $\|f\|_k <\infty$. A straightforward examination of the definitions reveals that $f$ is then $\mu$-harmonic on $\GG$, and $f \big|_\HH \equiv \tilde f$. Thus
$f \in \HF_k(\GG,\mu)$.

To recap:  for every $\tilde f \in \HF_k(\HH,\mu_\HH)$ the extension $f(x) : = \E_x[ \tilde f(X_\tau)]$ is
a function $f \in \HF_k(\GG,\mu)$.

{\bf Step II: Restriction.}
To show that this is indeed a unique extension, it suffices to show that if $f \big|_\HH \equiv 0$
and $f \in \HF_k(\GG,\mu)$ then $f \equiv 0$ on all of $\GG$.
Indeed, if $f \in \HF_k(\GG,\mu)$ then $(f(X_t))_t$ is a martingale.

Also, $(f(X_{\tau \mn t}))_t$ is uniformly integrable:
\begin{align*}
\E_x [ |f(X_{\tau \mn t})| ] & \leq C \cdot \E_x [ |X_{\tau \mn t}|^k ]
= C \cdot \E_x [ |X_\tau |^k \1{\tau \leq t} ] + C \cdot \E_x [ |X_t|^k \1{ \tau >t } ] \\
& \leq C \cdot \E_x [|X_\tau|^k] + C \cdot \E_x [ |X_t|^k \1{ \tau > t} ] .
\end{align*}
We have already seen above that $\E_x [ |X_\tau|^k ] \leq C_k |x|^k$ for some $C_k>0$.
Also, Lemma \ref{lem: moments of Xt} guaranties that because $\tau$ has an exponential tail,
\begin{align*}
\big( \E_x [ |X_t|^k \1{ \tau >t } ] \big)^2 & \leq \Pr_x [ \tau >t ] \cdot \E_x [ |X_t|^{2k} ]
\leq e^{-ct} \cdot C_{2k} \cdot (|x|+t)^{2k} .
\end{align*}
Thus, we obtain that
$$ \sup_t \E_x [ |f(X_{\tau \mn t})| ] \leq C \cdot C_{k} \cdot |x|^k + C \cdot C_{2k} \cdot \sup_t e^{-ct} (|x|+t)^k
= O(|x|^k) . $$
So $(f(X_{\tau \mn t}))_t$ is uniformly integrable indeed.

Thus, we may apply the Optional Stopping Theorem  \cite[\S 5.7]{Durrett} to obtain $\E_x [ f(X_{\tau}) ] = f(x) $.
Hence, if $f(x) = 0$ for all $x \in \HH$, then $f(X_\tau) = 0$ and so $f \equiv 0$ on all of $\GG$.

This shows that the linear map $f \mapsto f \big|_\HH$ is injective on $\HF_k(\GG,\mu)$.
\end{proof}

%% file: HF_solv_reduction_to_2x2.tex
\subsection{The reduction}
\label{scn: reduction}

We will now complete the proof of Theorem \ref{thm: main thm} assuming Theorem \ref{thm:2x2}. 

The dimension of $\HF_k$ can only decrease when passing to a quotient group: Whenever $\GG/N$ is a  quotient of $\GG$,
we have
$\dim \HF_k(\GG/N,\mu \circ \pi^{-1} ) \le \dim \HF_k(\GG,\mu)$ where $\pi:\GG \to \GG/N$ is the canonical projection.
Also, by Proposition \ref{prop:restriction_of_harmonic}, restricting to a finite index subgroup does not change the dimension of $\HF_k$. 
Thus Theorem \ref{thm: main thm} follows from the following proposition.

\begin{prop}[see \cite{Breuillard06}]
\label{prop:reduction_to_affine}
 Let  $\GG$  be a finitely generated solvable group which is not virtually nilpotent.
Then there exists a quotient  $\GG/N$  of $\GG$ which  is not virtually nilpotent, and has a finite index subgroup $\HH < \GG/N$ which embeds in $\AA(\FF)$ for some field $\FF$.
\end{prop}
This precise reduction is carried out  in \cite{Breuillard06}. We provide a few details.
The argument is based on the following theorem of Groves:

\begin{thm}[Groves \cite{Groves78}] \label{thm: Groves}
Let $G$ be a finitely generated solvable group that is just non virtually nilpotent
(that is, for any non-trivial normal subgroup $N \ns G$ we have that $G/N$ is virtually nilpotent).
Then there exists a finite index subgroup $[G:H] < \infty$ such that $H$ is isomorphic to
a subgroup of the affine group over a field $K$.
\end{thm}

 A proof  of Theorem \ref{thm: Groves} also appears in  \cite[Section 4]{Breuillard06}.

\begin{proof}[Outline of proof of Proposition \ref{prop:reduction_to_affine}]
We assume that $\GG$ is solvable but not virtually nilpotent, and $\mu$ a courteous measure on $\GG$. 

Every finitely generated non virtually nilpotent group has a just non virtually nilpotent quotient
(see \eg Claim 2 in the beginning of Section 5 of \cite{Breuillard06}).

So let $\GG/N$ be a  just non virtually nilpotent quotient of $\GG$.
Since $\GG$ is finitely generated and solvable, so is $\GG/N$.
By Theorem \ref{thm: Groves}  $\GG/N$ has  a finite index subgroup $\HH$ that is isomorphic
to a subgroup of the affine group over some field $\FF$.
\end{proof}

We now complete the section with

\begin{proof}[Proof of Theorem \ref{thm: main thm} assuming Theorem \ref{thm:2x2}]
Let $\GG$ be a finitely generated virtually solvable group, $\mu$ a courteous measure on $\GG$ and
assume that $\GG$ is not virtually nilpotent.

By Proposition \ref{prop:reduction_to_affine},
there exists $N \ns \GG$ and a finite index subgroup $\HH < \GG/N$ such that $\GG/N$ is not virtually nilpotent 
and $\HH$ is isomorphic
to a subgroup of $\AA(\FF)$ for some field $\FF$.
Since $\HH$ is finite index in a quotient of $\GG$, we have by Proposition \ref{prop:restriction_of_harmonic}
that $\dim \HF_1(\HH, \nu) \leq \dim \HF_1(\GG,\mu)$, for some courteous measure $\nu$
(obtained by projecting $\mu$ from $\GG$ to $\GG/N$ and then taking the induced hitting measure on $\HH$).
Because $\GG/N$ is not virtually nilpotent, it cannot be that $\HH$ is virtually abelian.
Theorem \ref{thm:2x2} now tells us that $\dim \HF_1(\GG,\mu) \geq \dim \HF_1(\HH,\nu) = \infty$.
\end{proof}

%% file: HF_solv_RW_real.tex
\section{Random walks on the reals}
\label{scn: RW on reals}

In this section we collect probability estimates and results regarding random walks on the real line. These estimates will serve us in the following section to prove Theorem \ref{thm:2x2}.
As this is not the main focus of this paper, we omit proofs for some standard statements.

In the following, $(Y_t)_t$ is a sequence of real valued
random variables such that $Z_t : = (Y_{t}-Y_{t-1})_t$ are i.i.d.\ symmetric
random variables of mean $0$. Thus, $(Y_t)_t$ is a martingale.  
We assume that the random variables $Z_t$ have an exponential tail; that is, there exists $\eps>0$ such that 
$\E[e^{\eps |Z_t|}] < \infty$.

For any set $A \subset \R$, $\tau_A$ is the \define{hitting time} of $A$ and $\sigma_A$ the \define{exit time} of $A$, \ie
$$ \tau_A = \inf \set{ t \ge 0 \ : \ Y_t \in A} \AND \sigma_A = \tau_{\R \setminus A} . $$
For $r>0$ set $\tau_r = \tau_{[-r,r]}$ and $\sigma_r = \sigma_{[-r,r]}$.
$\Pr_y, \E_y$ denote the probability measure and expectation conditioned on $Y_0=y$.

\subsection{Standard lemmas}
The following three lemmas are  relatively standard, and we omit the proofs.

\begin{lem}
\label{lem: exit time}
There exist constants $c,C>0$ (depending only on the distribution of $Z_1$)
such that for any $r>1$ and any $|y| \leq r$,
$$ \E_y [ \sigma_r ] \leq C r^2 \AND \Pr_y [ \sigma_r > t ] \leq e^{-c t / r^2 } . $$
\end{lem}

\begin{lem}
\label{lem: big jump}
There exist constants $C, \eps>0$ such that for  any $r>0$ and $|y| < r$, for all $z>0$,
$$ \Pr_y [ \exists \ t \leq \sigma_r  \ : \ |Z_t| > z ] \leq C r^3 e^{- \eps z} + e^{-r} . $$
\end{lem}

We also need  an estimate for the probability that $(Y_t)_t$ exits an interval from the left (or from the right).

\begin{lem}
\label{lem: Green function}
There exists a universal constant $\delta>0$ such that for all $y \in (-r,r)$,
as $r \to \infty$,
$$ \Pr_y [ \tau_{(r,\infty) } < \tau_{(-\infty,-r)} ] = \frac{y+r}{2r} \cdot (1+O(r^{-\delta})) . $$
\end{lem}

For random walks with bounded jumps, the proof is quite easy and standard.
Without the bounded jumps assumption,  there are some technical details  to deal with. 
This can be handled for instance by bounding the probability of a very large jump occurring before exiting the interval, 
using Lemma \ref{lem: big jump}. 
Again, we omit the  proof, pointing the reader for instance to \cite{MR2466937} for some results and proofs of similar flavor.

\subsection{Specific lemmas}

The next lemma is somewhat specifically tailored for our application.
It asserts that the random walk is very unlikely to spend too much time near the endpoints of an interval before exiting, even if 
we condition on the side from which the random walk exists. We therefore include a full proof.  
As mentioned, our bounds are not optimal, and we focus on brevity.
\begin{lem}
\label{lem: time in an interval}
\label{lem: left visits}
Let $V_m$ be the time spent by $(Y_t)_t$ in the interval $[0,m]$ until exiting $[0,r]$; that is,
$$ V_m = \sum_t \1{ Y_t \in [0,m] , \sigma_{[0,r]} > t} = \sum_{t=0}^{\sigma_{[0,r]}-1} \1{ Y_t \in [0,m] } . $$
Let $B = \set{ \tau_{(r,\infty) } < \tau_{(-\infty,0) } }$.
There exists a constant $c>0$
(depending only on the distribution of $Z_1$)
such that for all $0 < m < r$   and  $0 < y < r$ the following holds:
$$ \Pr_y [ V_m > v ] \leq e^{-c v /m^2 } \AND \Pr_y [ V_m > v \ , \ B ] \leq
e^{-c v /m^2 } \cdot \sup_{x \in [0,m] } \Pr_x[ B ] . $$
\end{lem}

\begin{proof}
Using Lemma \ref{lem: exit time} choose $C>0$ (independent of $m$)
be such that $\Pr_0 [ \sigma_{m} \leq C m^2 ] > \tfrac12$.
For every $y \in [0,m]$, after translating by $y$ and using the fact that the walk is symmetric,
\begin{align*}
\Pr_y [ \tau_{(-\infty,0)} \leq C m^2 ] & = \Pr_0 [ \tau_{(-\infty,-y)} \leq C m^2 ]
\geq \Pr_0 [ \tau_{(-\infty, -y)} < \tau_{(y,\infty)} , \sigma_{y} \leq C m^2 ] \\
& = \tfrac12 \Pr_0 [ \sigma_{y} \leq C m^2 ] \geq \tfrac12 \Pr_0 [ \sigma_m \leq C m^2 ] \geq \tfrac14 ,
\end{align*}
the last line following since $\sigma_{y} \leq \sigma_m$ a.s.
Since this is uniform over all $y \in [0,m]$ we have that for any $0 < y < r$,
taking $s: = \lceil C m^2 \rceil$,
\begin{align*}
\Pr_y [ V_m > v ] & \leq \sup_{z \in [0,m] } \Pr_z [ \tau_{(-\infty,0)} > s ] \cdot \sup_{x > 0} \Pr_x [ V_m > v-s ] \\
& \leq \tfrac34 \cdot \sup_{x > 0 } \Pr_x [ V_m > v-s ] \leq \cdots \leq \sr{ \tfrac34 }^{ \lfloor v/s \rfloor } .
\end{align*}

The proof of the second assertion is similar.
\begin{align*}
\Pr_y [ V_m > v \ , \ B ] & 
& \leq \tfrac34 \cdot \sup_{x > 0 } \Pr_x [ V_m > v-s  \ , \ B ] \leq \cdots 
\leq \sr{ \tfrac34 }^{ \lfloor v/s \rfloor - 1} \cdot \sup_{x>0} \Pr_x [ V_m \geq 1 \ , \ B ]  .
\end{align*}
The strong Markov property at time $\tau_{[0,m]}$ gives that
$$ \sup_{x > 0 } \Pr_x [ V_m \geq 1 \ , \ B ] \leq \sup_{x \in [0,m]} \Pr_x [ B ] . $$
\end{proof}

For the next lemma we need the notion of a maximal separated subset.
Given an interval $I \subset \R$ and $r >0$ let
$\MSep_r(I)= \MSep_r((Y_t)_t,I)$ denote the maximal cardinality of a $1$-separated subset of 
$I\cap \{ Y_0, Y_1,\ldots, Y_{\sigma_r-1} \}$; that is
$$ \MSep_r(I) = \max \big\{ |A| \ : \ A \subset I\cap \{ Y_0, Y_1,\ldots, Y_{\sigma_r-1} \} \ \textrm{ and } 
\ \forall \ a \neq a' \in A , \  |a-a'|    \geq 1 \big\} . $$

\begin{lem}\label{lem:Msep_prob_small}
For any $q>0$ there exists a constant $C >0$ such that 
for all $y \leq 0$ with $r > 2\max \set{-y,q}$, and any $n \leq \sqrt{r}$,
$$\Pr_y [ \MSep_r((-\infty,-q)) \leq n ] <  \frac{C(n+1)}{r} . $$
\end{lem}

\begin{proof}
Define inductively:  $T_0=0$ and $m_0 = Y_0$.  For $k>0$ define
$$ T_{k} : = \inf \set{ t > T_{k-1}  \ :  \  Y_t \leq m_{k-1} - 1 }  $$
and $m_k : = Y_{T_k}$.
So $(T_k)_k$ are the successive times the random walk $(Y_t)_t$ passes below its minimum by at least $1$.
Note that by definition $\set{m_0,m_1,\ldots,m_k}$ is a $1$ separated subset of $(-\infty,0]$,
and since $m_k \leq m_{k-1}-1$ it must be that for $\ell = \lceil q \rceil+1$, the set $\set{ m_{\ell} , \ldots, m_{\ell+n} }$ 
is a $1$-separated subset of $(-\infty, -q)$ of size $n+1$. 
Thus, the event $\set{ T_{\ell+n} < \sigma_r }$ implies the event $\set{ \MSep_r((-\infty,-q)) > n }$.

Now, let $\Ee = \set{ \forall \ t \leq \sigma_r \ , \ |Z_t| \leq r/2n }$.  
By Lemma \ref{lem: big jump}, $\Pr_y [ \Ee^c ] \leq  C r^3 e^{- \eps r / 2n}$ for some constants $C,\eps$ 
that depend only on the distribution of $Z_t$.  
By adjusting the constant in the statement of the lemma we may assume without loss of generality 
that $\ell < n$.  Thus, $(\ell+n) \cdot \tfrac{r}{2n} < r$.  Hence
the event $\set{ T_{\ell+n} \geq \sigma_r } \cap \Ee$ implies the event 
$\set{ T_{\ell+n} \geq \sigma_r = \tau_{(r,\infty)} }$. 
Since $T_0 < \sigma_r$ a.s., the event $\set{ T_{\ell+n} \geq \sigma_r } \cap \Ee$ implies that there exists
$0< k \leq \ell+n$ such that $T_{k-1} < \sigma_r$ and $T_k \geq \sigma_r = \tau_{(r,\infty)}$.
The probability of this can be bounded by the strong Markov property at time $T_{k-1}$ 
and Lemma \ref{lem: Green function},
\begin{align*}
\Pr_y [ T_{k-1} < \sigma_r \ , \ T_k \geq \sigma_r  = \tau_{(r,\infty)} ] & \leq \sup_{y \in [-r,-(k-1)] } 
\Pr_y [ T_1 \geq \sigma_r = \tau_{(r,\infty)} ] 
\\
& \leq \sup_{y \in [-r,-(k-1)] } \Pr_y [ \tau_{(-\infty, y-1]} > \tau_{(r,\infty) } ] 
\leq \frac{C}{r} .
\end{align*}
Thus, 
\begin{align*}
\Pr_y [ \MSep_r((-\infty,-q)) \leq n ] & \leq \Pr_y [ \Ee^c ] + 
\sum_{k=1}^{\ell+n} \Pr_y [   T_{k-1} < \sigma_r \ ,\ 
T_k \geq \sigma_r= \tau_{(r,\infty) } ] 
\\
& \leq C  r^3 e^{-\eps r / 2n } + \frac{C(\ell+n)}{r} .
\end{align*}
The lemma follows since for $n \leq \sqrt{r}$ we have that $r^3 e^{-\eps r / 2n} \leq C' r^{-1}$ for some constant $C'>0$.
\end{proof}

%% file: HF_solv_2x2.tex
\section{A positive harmonic function for subgroups of $\AA(\FF)$}
\label{sec:2x2}
In this section we prove Theorem \ref{thm:2x2}, using the random-walk estimates from the previous section.

We make the following obvious identification:

$$\AA(\FF) : = \left\{
\clamt{c}{\lambda}
\ : \
c \in \FF \ , \  \lambda \in \FF^\times
 \right\}.$$

 For $x =
\clamt{c}{\lambda}
\in\AA(\FF)$, we denote:
\begin{equation}\label{eq:lambda_alpha_c_dec}
 c(x) = c , \quad \lambda(x)= \lambda
\end{equation}

Let $\GG < \AA(\FF)$ be as in the statement of Theorem \ref{thm:2x2}.
Some  reductions will be useful.
\begin{lem}
\label{lem:Af_not_nil}
Let $\FF$ be a field and suppose $\GG$ is a finitely generated subgroup of $\AA(\FF)$. If the set  $\lambda(\GG) = \{ \lambda ~:~
\clamt{c}{\lambda} \in \GG\}$ is contained in the group of roots of unity of $\FF^{\times}$ then $\GG$ is virtually abelian.
\end{lem}
\begin{proof}
If $\lambda(\GG)$ is contained in the group of roots of unity of $\FF^{\times}$ then it  is a finite group, and the kernel of the homomorphism $\lambda:\GG \to \FF^\times$ is an abelian group, consisting of elements of the form $\clamt{c}{1}$.
\end{proof}

The following simple version of Kronecker's Theorem is used in the proof of the Tit's Alternative \cite{tits_1972}. See for instance \cite[Section $2$]{glander_breuillard_topological_tits} for a generalization and brief discussion.

\begin{lem} \label{lem: Kronecker}
Let $\FF$ be a finitely generated field and let $\alpha \in \FF$ be an element
which is not a root of unity.
Then there exists a local field $K$  with absolute value $|\cdot|$ and an embedding of fields $\iota : \FF \to K$, such that $|\iota(\alpha)|>1$.
\end{lem}

In our case, since $\GG$ is a finitely generated subgroup
of the affine group over $\FF$, we  can assume without loss of generality that $\FF$ is a field generated by
$\{ \lambda, c ~:~ \clamt{c}{\lambda} \in S\}$, where $S$ is a finite generating set for $\GG$.
Keep in mind that $\FF$ can be of positive or zero characteristic.
Thus, there is an embedding $\iota:\FF \to K$, where $K$ is a local field such that  $|\iota(\lambda)|>1$
for some $\clamt{c}{\lambda} \in \GG$,
and $|\cdot|$ is the absolute value on $K$.  Since $\iota$ is an embedding of fields, it induces an embedding  of groups
$\iota:\AA(\FF) \to \AA(K)$. 

With the above considerations in mind,  from now on we assume  $\FF$  is a  local field with absolute value $|\cdot|$;
the group $\GG \le \AA(\FF)$
is a finitely generated, non-abelian countable group such that $|\lambda| >1$ for some $\clamt{c}{\lambda} \in \GG$;
and  $\mu$ is a courteous probability measure on $\GG$.
Furthermore, we assume without loss of generality that there is an element $x \in \GG$ of the form $x= \clamt{0}{\lambda}$ with $|\lambda| >1$. Indeed, If $\clamt{c}{\lambda} \in \GG$ with $|\lambda|>1$,
we get an element of the correct form by conjugating $\GG$ with $\clamt{(1-\lambda)^{-1} c}{1}$.

\begin{lem}
There exist a finite index subgroup $\HH < \GG$ and an element $z \in \HH$ of the form 
\begin{equation}\label{eq:z_pos_measure}
z =  \clamt{c}{1}  \qquad  c \ne 0
\end{equation}
such that the hitting measure satisfies $\mu_\HH(z) >0$  	
\end{lem}

\begin{proof}
Because $\GG$ is not abelian, it contains a non-trivial commutator.
Furthermore, since $\mu$ is adapted,
we can find non-commuting elements $a,b \in \GG$
so that $\mu(a)>0$ and $\mu(b)>0$.
Because $a$ and $b$ are non-commuting, it follows that $a \ne 1$ and $b \ne 1$ so we can assume that $\lambda(a) \ne  1$ and $\lambda(b) \ne 1$, otherwise we can conclude the proof with  $\HH=\GG$ and $z=a$ or $z=b$.
By possibly replacing $a$ with $a^{-1}$ we can further assume that $\lambda(ab) \ne 1$.
Let $z:=[a,b]$. It follows that indeed $z =  \clamt{c}{1}$ with $c\ne 0$.
  Because $\lambda(\GG) \subset \FF^\times$ is a finitely generated abelian group it is residually finite, so there exists
a finite index subgroup $\Lambda_0 < \lambda(\GG)$ so that $\lambda(a),\lambda(b),\lambda(ab) \not\in \Lambda_0$. Let $\HH = \lambda^{-1}(\Lambda_0)$, then $\HH < \GG$ is a finite index subgroup and $z \in \HH$.
By the construction of $\HH$, we have $a,ab,aba^{-1} \not \in \HH$.
It follows that $\mu_\HH(z) \geq 
\mu(a^{-1})\mu(b^{-1})\mu(a)\mu(b) >0$
\end{proof}

Thus, by replacing $(\GG,\mu)$ with $(\HH,\mu_\HH)$ we further assume without loss of generality that $\mu(z) >0$ for some $z \in \GG$ satisfying \eqref{eq:z_pos_measure}.

We introduce a bit more notation. Let
\begin{equation}
\rho(x) = -\log|\lambda(x)| .
\end{equation}

$(X_t)_t$ denotes a discrete time random walk on $\mathbb{G}$ such that the jumps $S_t = X_{t-1}^{-1} X_t$ are identically distributed with distribution $\mu$.
$\Pr_x, \E_x$  denote the probability measure and expectation
of random walks with $X_0= x$. Specifically, $\Pr_x [ X_{t+1} = y \ | \ X_t = z ] = \mu(z^{-1} y) $.

As in Section \ref{scn: RW on reals}, for $A \subset \RR$, we use the notation
$$ \tau_A = \inf \set{ t \geq 0 \ : \ \rho(X_t) \in A } \AND \sigma_A = \inf \set{ t \geq 0 \ : \ \rho(X_t) \not\in A } . $$
For $r >0$ we use the abbreviations  $\tau_r = \tau_{[-r,r]} , \sigma_r = \sigma_{[-r,r]}$.

We define a function $f:\GG \to [0,\infty)$ as follows:
\begin{equation}\label{eq:PLHF_G}
f(x)= \lim_{k \to \infty} r_k \Pr_x [|c(X_{\sigma_{r_k}})| < 3 ],
\end{equation}
where 
$(r_k)_k$ is some increasing sequence of integers for which the limit in \eqref{eq:PLHF_G} exists.
We will prove Theorem \ref{thm:2x2} by showing that $f$ given by \eqref{eq:PLHF_G} satisfies all the requirements.

\subsection{Preliminary lemmas}

There are fairly straightforward bounds on  $|\rho(x)|$ and $|c(x)|$ which we now note, omitting the simple proof.
\begin{lem}\label{lem:alpha_distance}
There exists a constant $K >0$ so that for any $x \in \GG$

$$ |\rho(x)| \le K |x| \mbox{ and }
 |c(x) | \le e^{K|x|}$$
\end{lem}
%

The following lemma is crucial for proving that the function constructed has at most linear growth, and for proving it is non-constant.
\begin{lem}
\label{lem: alpha difference}
There exists a constant $K>0$ (depending only on $\GG$ and $\mu$)
such that for all $x \in \GG$ and $r>0$ with $0 < \rho(x) < r$,
$$ \Pr_x [ |c(X_{\sigma_r}) - c(X_0) | > 2 \ , \ \tau_{(r,\infty)} < \tau_{(-\infty,0)} ] \leq \frac{K}{r} . $$
\end{lem}

\begin{proof}

By Lemma \ref{lem:alpha_distance}, because $\mu$ is smooth, there exists $\eps>0$ so that
$\E[ |c(S_1)|^{2\eps} ] = \E [ |c(X_1)|^{2\eps} ] <\infty$.
Let $B = \set{ \tau_{(r,\infty)} < \tau_{(-\infty,0)} }$ be the event that
the walk $(\rho(X_t))_t$ exits $[0,r]$ from the right.

Define
$$ A_m = \sum_{t=0}^{\sigma_{[0,r]} -1 } |c(S_{t+1}) | \1{ \rho(X_t) \in [m,m+1) } .$$
The quantity $A_m$ will be used to control the contribution to $|c(X_{\sigma_r}) - c(X_0) |$ from increments of $c$ at the times $\rho(X_t)$ is in the interval $[m,m+1)$, before exiting $[0,r]$.

We have that for any $x$ with $\rho(x) >0$, by Lemma \ref{lem: Green function} and the Markov property
at time $t$,
\begin{align*}
\E_x [ |c(S_{t+1})|^{\eps} & \ind{ \{ t < \sigma_{r} \} } \1{ \rho(X_t) \in [m,m+1) } \ind{B} ] \\
& \leq \sum_s \mu(s) |c(s)|^{\eps} \cdot \sup_{y \ : \ \rho(y) < m+1 } \Pr_{ys} [ B ] \cdot
\Pr_x [ \rho(X_t) \in [m,m+1) , t < \sigma_{[0,r]}  ] \\
& \leq C \cdot \sum_s \mu(s) |c(s)|^{\eps} \cdot \frac{m+1 + |\rho(s)|}{r} \cdot
\Pr_x [ \rho(X_t) \in [m,m+1) , t<\sigma_{[0,r]} ] \\
& = \frac{C}{r} \cdot \E [ |c(X_1)|^{\eps} (m+1+ |\rho(X_1)| ) ] \cdot  \Pr_x [ \rho(X_t) \in [m,m+1) , t<\sigma_{[0,r]} ] .
\end{align*}
By Lemma \ref{lem:alpha_distance}, because $\mu$ is smooth  $\E [ |\rho(X_1)|^2 ] < \infty$, so
using Cauchy-Schwarz,
\begin{align*}
\E [ |c(X_1)|^{\eps} & (m+1 + |\rho(X_1)| ) ] \\
&  \leq (m+1) \E [|c(X_1)|^{\eps} ] + \sqrt{ \E [ |c(X_1)|^{2\eps}]
\cdot \E [ |\rho(X_1)|^2 ] }  \leq C m
\end{align*}
where $C>0$ is a constant that depends only on $\E [ |c(X_1)|^{2\eps} ] , \E [ |\rho(X_1)|^2 ]$.
Thus, for some constant $C>0$, by Markov's inequality,
\begin{align}
\label{eqn: alpha jump}
\Pr_x [ |c(S_{t+1})| > \beta & \ , \  t < \sigma_{r} \ , \ \rho(X_t) \in [m,m+1) \ , \  B ] =\nonumber \\
\Pr_x [ |c(S_{t+1})|^{\epsilon} > \beta^{\epsilon} & \ , \  t < \sigma_{r} \ , \ \rho(X_t) \in [m,m+1) \ , \  B ] \leq \nonumber \\
&  \leq
\frac{Cm}{\beta^{\epsilon} r} \cdot \Pr_x [ \rho(X_t) \in [m,m+1) , t<\sigma_{[0,r]} ]
\end{align}

Define
$$ V_m = \sum_{t=0}^{\sigma_{[0,r]}-1} \1{ \rho(X_t) \in [m,m+1) } , $$
the number of visits to $[m,m+1)$ by the walk $(\rho(X_t))_t$, before exiting $[0,r]$.
Lemma \ref{lem: left visits} together with Lemma \ref{lem: Green function}
tell us that for some constant $c>0$,
$$ \Pr_x [ V_m > v ] \leq e^{-c v / m^2 }  \AND \Pr_x [ V_m > v \ , \ B ] \leq e^{-c v / m^2 } \cdot \frac{Cm}{r} . $$
So we may choose $C>0$ large enough such that for all $m$,
$$ \Pr_x [ V_m > C m^3 \ , \ B ] \leq \frac{1}{r} \cdot e^{-m}  \AND \E_x[V_m] \leq C m^2 . $$

Summing \eqref{eqn: alpha jump} over $t$, we have that
\begin{align*}
\Pr_x [ \exists \ t < \sigma_{r} & \ , \  |c(S_{t+1})| > \beta \ , \ \rho(X_t) \in [m,m+1) \ , \  B ] \\
& \leq \E_x [ V_m ] \cdot \frac{Cm}{\beta^{\eps} r} \leq \frac{C' m^3}{\beta^{\eps} r} .
\end{align*}
Taking $\beta = C^{-1} m^{-3} (e/2)^m$,
\begin{align*}
\Pr_x [ A_m e^{-m} > 2^{-m} \ , \ B ] & \leq
\Pr_x [ A_m e^{-m} > 2^{-m} \ , \ V_m \leq Cm^3 \ , \ B ] + \Pr_x [ V_m > Cm^3 \ , \ B ] \\
& \leq \frac{C' m^3 }{\beta^\eps r}  + \frac{1}{r} e^{-m} \leq \frac{1}{r} e^{-\delta m} ,
\end{align*}
for some constant $\delta = \delta(\eps) >0$.
Summing over $m \geq 0$ we obtain
$$ \Pr_x [ \exists \ m \geq 0 \ , \ A_m e^{-m} > 2^{-m} \ , \ B ] \leq \frac{C}{r} . $$

Since $c(xy) = c(x) + \lambda(x) c (y)$, we have
$$ c(X_{\sigma_r}) - c(X_0) = \sum_{t=0}^{\sigma_r-1} c(S_{t+1}) \cdot \lambda(X_t) , $$
and so
$$ |c(X_{\sigma_r}) - c(X_0)| \leq \sum_{t=0}^{\sigma_r-1} |c(S_{t+1})| \cdot e^{- \rho(X_t) }
\leq \sum_m A_m e^{-m} . $$
On the event $B$ we have that $A_m=0$ for all $m<0$.
So the event $B \cap \set{ \forall \ m \geq 0 \ , A_m e^{-m} \leq 2^{-m} }$ implies that
$$  |c(X_{\sigma_r}) - c(X_0)| \leq \sum_{m \geq 0} 2^{-m} = 2 . $$
Thus, 
for any $x \in \GG$ with $\rho(x) > 0$,
\begin{align*}
 \Pr_x [ |c(X_{\sigma_r}) - c(X_0) | > 2 , B ] \leq \frac{C}{r} .
\end{align*}
\end{proof}

As in the previous section, just before Lemma \ref{lem:Msep_prob_small}, 
we denote the maximal cardinality of a $1$-separated subset of
$\set{ \rho(X_{t}) \ : \ t < \sigma_r } \cap I$ by $\MSep_r(I)= \MSep_r((\rho(X_t))_t,I)$.

The following is a key lemma for proving the function $f$ is non-constant.

\begin{lem}
\label{lem: comm argument Msep}
There exist $C,q,\epsilon >0$, depending only on $\GG,\mu$, such that 
$$ \Pr_x \left[ |c(X_{\sigma_r})  | < 3 \ | \ \MSep_r((-\infty,-q)) = n \right] \leq C e^{-\epsilon n} . $$
for all $ n \in \Z_+$, $r >0$ and 
 $x \in \GG$.

\end{lem}
\begin{proof}
By adjusting the constant in the statement of the lemma we may assume without loss of generality that $n \geq 1$.
Let $z \in \GG$ satisfy \eqref{eq:z_pos_measure} with $\mu(z)>0$. Such an element exists by our assumptions.
Choose $q >0$ which satisfies
\begin{equation}\label{eq:q_cond}
  q >  \log\left(\frac{3}{|c(z)|(1-\sum_{k=1}^\infty e^{-k})} \right), 
\end{equation}
and let $I = (-\infty,-q)$.

Consider the set $R = \set{ \rho(X_{t}) \ : \ t < \sigma_r } \cap I$.
Define the event
\begin{equation}\label{eq:Ee_n_Msep}
\Ee_n = \set{\MSep_r(I) = n}
\end{equation}
Note that the event $\Ee_n$ is measurable with respect to the (set-valued) random variable $R$.

Assume we are in  the event $\Ee_n$.
Let $A=A(R) \subset I\cap R$ be a $1$-separated set of size $n$.
Formally, $A$ is a set-valued random variable which is measurable with respect to the random variable $R$, 
and which on the event $\Ee_n$ is a.s.\ a $1$-separated subset $A \subset I\cap R$ of size $|A|=n$.

For $\rho \in R$ let $t_\rho = \inf \set{ t  \ge 0\ : \ \rho(X_{t}) = \rho }$
and  $T = \set{ t_\rho \ : \ \rho \in A }$.
For $t \ge 0$ define
$$ \xi_t =
\begin{cases}
1  & \textrm{ if } S_{t+1} = z \\
-1 & \textrm{ if } S_{t+1}= z^{-1} \\
0 & \textrm{ otherwise}
\end{cases}
$$
and let $T' = \set{ t \in T~:~ |\xi_t|=1}$. Note that $S_{t+1}= z^{\xi_t}$ for $t \in T'$.

Let $(\check S_t)_t$ denote the sequence obtained from $(S_t)_t$ by changing every occurrence of $z^{-1}$ at times in $T'$ to $z$:
$$\check S_{t} =
\begin{cases}
z & t \in T'\\
S_t & \mbox{otherwise}
\end{cases}
$$

We claim that:
\begin{align}
\label{conditional_N}
& \begin{array}{l}
\textrm{Conditioned on the event $\Ee_n$ and on $R$, the random variables} \\
\textrm{$(|\xi_{t_\rho}|)_{\rho \in A}$ are i.i.d.\ and $\Pr[ t_\rho \in T'\mid \ \rho \in A\ , R ] = 2\mu(z)$.}
\end{array}
\end{align}

For now let us proceed with the proof assuming \eqref{conditional_N}. Let
$$\Pi =  c(X_0) + \sum_{t=0}^{\sigma_r-1}\1{t \not\in T'}\lambda(X_{t})c(\check S_{t+1}).$$
Since $\lambda(z)=1$, it follows that $\lambda(\check S_j)= \lambda(S_j)$ so $\lambda(X_t)= \prod_{j=1}^{t-1} \lambda( \check S_{j})$ is measurable with respect to $(\check S_t)_t$, and hence also 
$\Pi$ is measurable with respect to  $(\check S_t)_t$.

Because $c(X_{\sigma_r}) = c(X_0) + \sum_{t=0}^{\sigma_r-1}\lambda(X_t)c(S_{t+1})$ we have:
$$
c(X_{\sigma_r}) = \Pi + \sum_{t \in T'}\lambda(X_{t})c(S_{t+1})
 = \Pi + c(z)\sum_{t \in T'}\lambda(X_{t}) \xi_t
$$

Note that for any $t \in T'$ we have that $\rho(X_t) \in R \subset I$ so $-\rho(X_t) > q$.

Also, for all $\rho \neq \rho' \in A$ we have that $|\rho - \rho' | \geq 1$, which implies that
for any $\xi'' \neq \xi' \in \set{-1,1}^{T'}$, there is some $\rho \in A$ such that
\begin{align*}
\big| c(z)\sum_{t \in T'}\lambda(X_{t})  (\xi''_t-\xi'_t)  \big| & \geq 2 |c(z)|e^{-\rho} \cdot \sr{ 1 - \sum_{k=1}^\infty e^{-k}  }
 >   6 ,
\end{align*}
where the last inequality holds by \eqref{eq:q_cond} because $e^{-\rho} > e^q$.
Thus, for any $\alpha \in \FF$ there is at most one possible vector $\xi \in \set{-1,1}^{T'}$ for which
$$ \big| \alpha + c(z) \cdot \sum_{t \in T'} \xi_t\lambda(X_{t})  \big| < 3 . $$

Note that the set $T'$ is measurable with respect to  $(\check S_t)_t$.
From the fact that $(S_t)_t$ are i.i.d and the definition of  $(\check S_t)_t$, it follows that conditioned on $(\check S_t)_t$
the distribution of $(\xi_t)_{t \in T'}$  is uniform on $\set{-1,1}^{T'}$.
Thus we have
\begin{align}
\label{eqn:N' bound}
\Pr_x [ | c(X_{\sigma_r})  | < 3 \ |  \Ee_n ] & = \E_x \Pr_x \Big[  \big|  \Pi+ c(z) \cdot
\sum_{t \in T'} \xi_{t} \lambda(X_{t})  \big| < 3
\ \Big| \ (\check S_t)_t \ , \Ee_n \Big] \nonumber \\
& \leq \E_x [2^{-|T'|} \ | \Ee_n ] \leq \Pr_x [ |T'|  < \mu(z) n \ | \Ee_n ] + 2^{-\mu(z) n} .
\end{align}

By \eqref{conditional_N}, conditioned on the event $\Ee_n$ and on the set $R$,
the distribution of $|T'|$ is binomial-$(n,2\mu(z))$.
Using a well known large deviation estimate for the binomial distribution

$$ \Pr_x [ |T'| < \mu(z) n \ | \Ee_n  ] \leq e^{-\epsilon n}. $$

for some constant $\epsilon=\epsilon(\mu(z))>0$.
So
combined with \eqref{eqn:N' bound} we have that
\begin{align}
\label{eqn:N bound}
\Pr_x [ |c(X_{\sigma_r}) | < 3  \ | \Ee_n ] \leq  e^{-\eps n} + 2^{-\mu(z) n} .
\end{align}

We now justify  \eqref{conditional_N}  by providing an alternative description of the process $(S_t)_t$.
The idea is that the process can be constructed  by sampling independent steps which are conditioned not to be $z$ or $z^{-1}$ and then ``spacing'' them with a geometrically distributed  number of steps each of which is equal to $z$ or $z^{-1}$ independently with equal probability. Formally:

Start with a sequence  $(\hat S_t)_t$ of  i.i.d  elements in $\GG$ each distributed according to $\mu$, conditioned  on the event
$\{ \hat{S}_t \not \in \{ z,z^{-1}\} \}$, and another sequence $(Z_t)_t$  of  i.i.d  elements in $\GG$ each distributed  so that

\begin{align}
\label{eqn:vec_C}
\Pr [ Z_{t} = z  ]& = \Pr [ Z_{t} = z^{-1} ] = \tfrac12 .
\end{align}

Let  $(T_j)_{j=1}^\infty$  be i.i.d.\
integer valued  random variables
with 
distribution $$\Pr [ T_j = k] =(2p)^{k}(1-2p) \mbox{ for }  k=0,1,2,\ldots,$$
where $p = \mu(z)$.
The processes $(T_j)_{j=1}^\infty$,  $(\hat S_t)_t$ and $(Z_t)_t$ are jointly
independent.

Now obtain a new process $(\tilde S_t)_t$ inductively as follows: Start with $L[1]=T_1$, $j[1] =1$ and  $k[1]=1$
and define inductively
$$
\left(L[t+1],j[t+1],k[t+1]\right) =\begin{cases}
(L[t] -1,j[t],k[t]+1) &  L[t] > 0\\
(T_{j[t]+1},j[t]+1,k[t]) & \mbox{otherwise}
\end{cases}
$$
For $t=1,2,\ldots$ let
$$
\tilde S_{t} = \begin{cases}
Z_{k[t]} &  L[t] >0\\
\hat S_{j[t]} & \mbox{otherwise}
\end{cases}
$$

It follows that $(\tilde S_t)_t$ are i.i.d, each distributed according to $\mu$. 
Thus, we may replace $(S_t)_t$ with $(\tilde S_{t})_t$.

Note that
the sets $R = \{ \rho(X_{t}) \ : \ t < \sigma_r  \} \cap I$ and  $A = A(R)$
are measurable with respect to $(\hat{S}_j)_j$.
For any $\rho \in T$
we have
$$\tilde S_{t_\rho} =
\begin{cases}
Z_{k[t_\rho]} & T_{j[t_\rho]} >0\\
\hat S_{j[t_\rho]} & \mbox{otherwise}
\end{cases}
$$
So $\tilde S_{t_\rho} \in \{z,z^{-1}\}$ if and only if $T_{j[t_\rho]} >0$.
Check that the process $(\hat{S}_j)_j$ is independent from $T_{j[t_\rho]}$ and $Z_{k[t_\rho]}$.
 From the above and the independence of $(T_j)_j , (\hat{S}_j)_j$ and $(Z_t)_t$  and the fact that
$\Pr[ T_j >0] = 2p$,  \eqref{conditional_N}  now follows.
\end{proof}

\subsection{Well defined and harmonic}

\begin{prop}
\label{prop: upper bound for negative rho}
There exists a constant $K>0$ such that
for all $x \in \GG$ and $r>0$ such that $-\tfrac{r}{2} <\rho(x) \leq 0$,
$$ \Pr_x [ |c(X_{\sigma_r}) | < 3 ] \leq \frac{K}{r} . $$
\end{prop}

\begin{proof}
Let $q >0$, $I = (-\infty,-q)$ and $\MSep_r(I)$  be as  in Lemma \ref{lem: comm argument Msep}.
$$
\Pr_x [ |c(X_{\sigma_r}) | < 3 ] \leq \sum_{n=0}^\infty \Pr_x[ |c(X_{\sigma_r}) | < 3 \ | \ \MSep_r(I) = n] 
\cdot \Pr_x[ \MSep_r(I) \leq n]
$$

By Lemma \ref{lem:Msep_prob_small}, for all $n \leq \sqrt{r}$ we have  $\Pr_x[ \MSep_r(I) \leq n] \le \frac{C(n+1)}{r}$.
So using Lemma \ref{lem: comm argument Msep},
$$
\Pr_x [ |c(X_{\sigma_r}) | < 3 ] \le  \sum_{n \leq \sqrt{r} } \frac{C (n+1)}{r} e^{-\epsilon n} 
+ \sum_{n > \sqrt{r} } e^{-\eps n} \leq \frac{K}{r} . $$
\end{proof}

\begin{prop}
\label{prop: upper bound general}
There exists a constant $C>0$, depending on $\GG$ and $\mu$,  such that for all $x \in \GG$ and $r > 2|\rho(x)|$,
$$ \Pr_x [ |c(X_{\sigma_r}) | < 3 ] \leq \frac{C \max \set{ \rho(x) , 1 }  }{r} . $$
\end{prop}

\begin{proof}
Let
$$ \Ee_1 = \set{ \exists t < \sigma_r \ , \  \rho(X_t) \in (-r/2,0)}
\qquad \textrm{and} \qquad  \Ee_2 = \set{  \tau_{(-\infty, 0]} < \tau_{(r,\infty)} }  .  $$
We have:
\begin{equation}\label{eq:c_small_E_1_E_2}
\Pr_x [ |c(X_{\sigma_r}) | < 3   ]  \le \Pr_x[ \Ee_2 \cap \Ee_1^c] + \Pr_x[  |c(X_{\sigma_r}) | < 3  \ , \  \Ee_1] +
\Pr_x [ \Ee_2^c]
\end{equation}

Note that because $\rho(X) > -\frac{r}{2}$ the event $\Ee_2 \cap \Ee_1^c$ implies that the random walk $(\rho(X_t))_t$ jumps across the interval $(-\frac{r}{2},0)$. So by Lemma \ref{lem: big jump}
\begin{align}
\label{eq:E_2_E1}
\Pr_x [  \Ee_2 \cap \Ee_1^c] & \le \Pr_x \big[ \exists \ t \leq \sigma_r  \ : \ |\rho(S_t)| > \tfrac{r}{2} \big] 
\leq C e^{- \eps r} ,
\end{align}
for some constants $C,\eps$.

By the strong Markov property and Proposition \ref{prop: upper bound for negative rho}, there exists $C>0$ such that
\begin{equation}\label{eq:c_Small_E_1}
\Pr_x [ |c(X_{\sigma_r}) | < 3  \ , \  \Ee_1]  \le \sup_{-r/2 < \rho(y) < 0 } \Pr_y [ |c(X_{\sigma_r})| < 3 ] \le \frac{C}{r} .
\end{equation}

By Lemma \ref{lem: Green function} there exists $C >0$ so that
\begin{equation}\label{eq:E_2_c}
\Pr_x [  \Ee_2^c]  \le \frac{\rho(x)  + C}{r}.
\end{equation}

The proof follows by applying  the bounds in \eqref{eq:E_2_E1}, \eqref{eq:c_Small_E_1} and \eqref{eq:E_2_c}
on the right hand side of  \eqref{eq:c_small_E_1_E_2}.
\end{proof}

\begin{prop}\label{prop:f_well_defined_harmonic}
There exists  an increasing sequence of integers $(r_k)_k$ for which the limit in \eqref{eq:PLHF_G} exists for all $x \in \mathbb{G}$. For such a sequence the function $f$ defined by the limit in \eqref{eq:PLHF_G} satisfies $f \in \HF_1(\GG,\mu)$. That is $f$ is $\mu$-harmonic and there exists some constant $C>0$
so that $|f(x)| \le C(|x|+1)$ for all $x \in \GG$.
\end{prop}

\begin{proof}
For $r \in \N$ let $f_r:\GG \to \R_+$ be given by
$$f_r(x) = r \cdot \Pr_{x} [ |c(X_{\sigma_r}) | < 3 ] \cdot \1{r > 2 |\rho(x)| }.$$

By Proposition \ref{prop: upper bound general}, $\sup_{r} f_r(x) < \infty$ for all $x \in \GG$, so by Arzel\`{a}-Ascoli  there exists a subsequence $(r_k)$ along which there is pointwise convergence. Let $f$ be this subsequential limit.

To see that $f$ is $\mu$-harmonic, note for $r>2|\rho(x)|$ we have 
$$ \Pr_x [ |c(X_{\sigma_r})| < 3 ] = \sum_s \mu(s) \Pr_{xs} [ |c(X_{\sigma_r})| < 3 ] . $$
Thus
$$ f(x) = \lim_{k \to \infty} r_k \cdot \Pr_x [ |c(X_{\sigma_{r_k}})| < 3 ] =
\sum_s \mu(s) \lim_{k \to \infty} r_k \cdot \Pr_{xs} [ |c(X_{\sigma_{r_k} })| < 3 ] = \sum_s \mu(s) f(xs) . $$

It remains to observe that by Proposition \ref{prop: upper bound general} and Lemma \ref{lem:alpha_distance},
$$\Pr_x [ |c(X_{\sigma_r}) | < 3 ]  \leq  \frac{C \max \set{ \rho(x) , 1 }  }{r} \leq \frac{C'(|x|+1)}{r} . $$
Multiplying by $r$ and taking  limits this proves that $f(x) \le C'(|x|+1)$.
\end{proof}

\subsection{Non-constant}

\begin{prop}
\label{prop: lower bound good}
There exist constants $C,\epsilon>0$ depending only on $\GG$ and $\mu$, such that for any $x \in \GG$ and $r>0$
with $|c(x) | < 1$ and $C < \rho(x) < r$,
$$ \Pr_x [ |c(X_{\sigma_r}) | < 3 ] \geq \frac{\epsilon \rho(x)-C}{r} . $$
\end{prop}

\begin{proof}
If  $|c(x) | < 1$ then
$$ \Pr_x [  |c(X_{\sigma_r}) | < 3 ] \geq \Pr_x [ |c(X_{\sigma_r}) - c(x) | \leq  2 ].$$
Let $B = \set{ \tau_{(r,\infty)} < \tau_{(-\infty,0]} }$.
We have
$$ \Pr_x [ |c(X_{\sigma_r}) - c(x) | \leq  2 ] \geq \Pr_x [ |c(X_{\sigma_r}) - c(x) | \leq  2 \ , \ B].$$
By Lemmas \ref{lem: Green function} and \ref{lem: alpha difference},
\begin{align*}
\Pr_x [ |c(X_{\sigma_r}) - c(x) | \leq 2 \ , \ B ] & \geq \Pr_x [ B ] - \frac{C}{r}
\geq \frac{\epsilon \rho(x) - C}{r} .
\end{align*}
\end{proof}

\begin{prop}\label{prop:f_non_constant}
The function  $f$ defined in \eqref{eq:PLHF_G} is non-constant.
\end{prop}

\begin{proof}
By one of our assumptions, there exists $x = \clamt{0}{\lambda} \in \GG$ with $|\lambda|>1$,
so $x^{-n} = \clamt{0}{\lambda^{-n}}$.
By Proposition \ref{prop: lower bound good},
$$\lim_{n \to \infty} f(x^{-n}) = +\infty,$$
 so  the function $f$ is unbounded and in particular non-constant.
\end{proof}

\subsection{Infinite dimensional orbit}

\begin{prop}
\label{prop: upper bound for orbit}
There exists a constant $C>0$ such that for all $x \in \GG$ with $|c(x)| > 5$,
and all $r >4 |\rho(x)|$,
$$ \Pr_x [ |c(X_{\sigma_r}) | < 3 ] \leq \frac{C}{r} . $$
\end{prop}

\begin{proof}
Let $B = \set{ \tau_{(r,\infty)} < \tau_{(-\infty,0]} }$.
Let $\Ee = \set{ \exists \ t < \sigma_r \ : \ \rho(X_t) \in (-r/2,0] }$.
By the strong Markov property
\begin{align*}
\Pr_x [ |c(X_{\sigma_r})| < 3 \ , \ B^c ] & \leq \Pr_x[ \Ee^c \ , \ B^c]
+ \sup_{y \ : \ \rho(y) \in (-r/2 , 0] } \Pr_y [ |c(X_{\sigma_r}) | < 3 ]
\end{align*}
By Proposition \ref{prop: upper bound for negative rho} the second term on the right hand side is bounded by
$\frac{C}{r}$.
Using Lemma  \ref{lem: big jump} (as in the proof of Proposition \ref{prop: upper bound general}), there are constants
$C,\eps >0$ such that $\Pr_x [ \Ee^c \ , \ B^c ] \leq C e^{-\eps r}$.

If $\rho(x) \leq 0$ then $\Pr_x[B] = 0$, so
$\Pr_x [ |c(X_{\sigma_r}) | < 3 \ , \ B ] = 0$.
If $\rho(x) > 0$ and  $|c(x)| > 5$  then
Lemma \ref{lem: alpha difference} tells us that for some $C>0$
$$ \Pr_x [ |c(X_{\sigma_r}) | < 3 \ , \ B ] \leq \Pr_x [ |c(X_{\sigma_r})-c(x)| > 2 \ , \ B ]  \leq \frac{C}{r} . $$
Altogether, for any $x \in \GG$ with $|c(x)|>5$,
$$ \Pr_x [ |c(X_{\sigma_r}) | < 3 ] = \Pr_x [  |c(X_{\sigma_r}) | < 3 \ , \ B ]
+ \Pr_x [  |c(X_{\sigma_r}) | < 3 \ , \ B^c ] \leq \frac{C}{r} .  $$
for some constant $C>0$.
\end{proof}

\begin{prop} \label{prop: positive HFs}
Let $f:\GG \to [0,\infty)$ be the function given in \eqref{eq:PLHF_G}.
There exist $(y_n)_n \subset \GG$ such that the family
$(f_n := y_n f)_n$ are infinitely many linearly independent functions.

Specifically, $\dim {\spn(\GG f)} = \infty$.
\end{prop}

\begin{proof}
By our assumptions on $\GG$, we have elements $z = \clamt{c}{1} \in \GG$ with $c \ne 0$,
and $x= \clamt{0}{\lambda} \in \GG$ with $|\lambda| >1$.
Choose $N$ large enough so that
$$|\lambda^N|\cdot (|\lambda^N|-1|)\cdot |c| > 5.$$

Let
$$y_n = x^{Nn}zx^{-Nn}= \clamt{\lambda^{Nn}c}{1} \mbox{ and } f_n = y_nf.$$

A simple calculation shows that for $m,n,j \in \mathbb{N}$:
$$f_n(y_mx^{-j})= f(y^{-n}y^mx^{-j})= f\left(\clamt{(\lambda^{Nm}-\lambda^{Nn})c}{\lambda^{-j}}\right)$$
If $1 \leq n  <  m$ then
$$\left|(\lambda^{Nm}-\lambda^{Nn})c \right| \geq  |\lambda^{Nn}|\cdot \big| |\lambda^{N(m-n)}| - 1\big| \cdot|c|  >  5 . $$
Using the symmetry between $n,m$, by Proposition \ref{prop: upper bound for orbit},
there exists a constant $C>0$ such that for any $n \neq m$,
we have $|f_n(y_m x^{-j}) | \leq C$ for any $j \in \mathbb{N}$.
On the other hand, by  Proposition \ref{prop: lower bound good} we have
$$\lim_{j \to \infty} f_n(y_nx^{-j}) = + \infty.$$
It follows that for if $\alpha_1,\ldots,\alpha_m \in \C$ and $\alpha_m \ne 0$ then
\begin{align*}
\left| \sum_{n=1}^m \alpha_n f_n(y_mx^{-j}) \right| & \geq |\alpha_m| \cdot |f_m(y_mx^{-j})|
- \sum_{n=1}^{m-1}|\alpha_n|\cdot  |f_n(y_mx^{-j})| \\
& \geq |\alpha_m| \cdot f_m(y_m x^{-j} ) - \sum_{n=1}^{m-1} |\alpha_n| \cdot C \to \infty \qquad \textrm{ as } j \to \infty.
\end{align*}
We conclude that the functions $(f_n)_n$ are indeed linearly independent.
\end{proof}


\begin{proof}[Proof of Theorem \ref{thm:2x2}]
Theorem \ref{thm:2x2} follows directly from the combination of Propositions \ref{prop:f_well_defined_harmonic},
\ref{prop:f_non_constant} and \ref{prop: positive HFs}.
\end{proof}